\documentclass[12pt]{amsart}
\usepackage[colorlinks=true,citecolor=black,linkcolor=black,urlcolor=blue]{hyperref}
\usepackage{amsmath}
\usepackage[english,  activeacute]{babel}
\usepackage[utf8]{inputenc}
\usepackage{amssymb}
\usepackage{amsthm}
\usepackage{graphics,graphicx}
\usepackage{enumerate}
\usepackage{ytableau}
\usepackage{array}
\usepackage{bm}
\usepackage{cite}
\usepackage{a4wide}
\usepackage{color, url}
\usepackage{float}
\usepackage{comment}
\usepackage{xcolor}
\usepackage{color, url}
\usepackage{float}
\usepackage{pgf, tikz}

\theoremstyle{plain}
\newtheorem{theorem}{Theorem}[section]
\newtheorem{proposition}[theorem]{Proposition}
\newtheorem{corollary}[theorem]{Corollary}
\newtheorem{lemma}[theorem]{Lemma}

\theoremstyle{definition}
\newtheorem{definition}[theorem]{Definition}
\newtheorem{example}[theorem]{Example}

\usepackage[usestackEOL]{stackengine}[2013-10-15]
\stackMath

\title{Combinatorics of $q$-Mahonian numbers of type $B$ and log-concavity}

\date{\today}

\begin{document}

\author[A. Kessouri]{Ali Kessouri}
\address{Department of Mathematics, University of Ferhat Abbas, Setif 1, LMFN Laboratory, Setif, Algeria}
\email{ali.kessouri@ensta.edu.dz}

\author[M. Ahmia]{Moussa Ahmia}
\address{University of Mohamed Seddik Benyahia, LMAM laboratory, Jijel, Algeria}
\email{moussa.ahmia@univ-jijel.dz;
ahmiamoussa@gmail.com}

\author[H. Arslan]{Hasan Arslan}
\address{Department of Mathematics, Faculty of Science, Erciyes University, 38039, Kayseri,
Turkey}
\email{hasanarslan@erciyes.edu.tr}

\author[S. Mesbahi]{Salim Mesbahi}
\address{Department of Mathematics, University of Ferhat Abbas, Setif 1, LMFN Laboratory, Setif, Algeria}
\email{salim.mesbahi@univ-setif.dz}

\newcommand{\nadji}[1]{\mbox{}{\sf\color{green}[Kessouri: #1]}\marginpar{\color{green}\Large$*$}} 

\newcommand{\ahmia}[1]{\mbox{}{\sf\color{magenta}[Ahmia: #1]}\marginpar{\color{magenta}\Large$*$}}

\begin{abstract}
This paper is a continuation of earlier work of Arslan \cite{Ars}, who introduced the Mahonian number of type $B$ by using a new statistic on the hyperoctahedral group $B_{n}$, in response to questions he suggested in his paper entitled "{\it A combinatorial interpretation of Mahonian numbers of type $B$}"  published in \url{https://arxiv.org/abs/2404.05099v1}. We first give the Knuth-Netto formula and  generating function for the subdiagonals on or below the main diagonal of the Mahonian numbers of type $B$, then its combinatorial interpretations by lattice path/partition and tiling. Next, we propose a $q$-analogue of Mahonian numbers of type $B$ by using a new statistics on the permutations of the hyperoctahedral group $B_n$ that we introduced, then we study their basic properties and their combinatorial interpretations by lattice path/partition and tiling. Finally, we prove combinatorially that the $q$-analogue of Mahonian numbers of type $B$ form a strongly $q$-log-concave sequence of polynomials in $k$, which implies that the Mahonian numbers of type $B$ form a log-concave sequence in $k$ and therefore unimodal.
\end{abstract}

\maketitle
\section{Introduction}
Permutation is one of the fundamental notions in combinatorics for problems of enumeration and discrete probabilities. It is used to define and study classic problems such as the magic square, the Latin square and Rubik's Cube. Permutations play an important role in group theory, determinant theory and the theory of symmetric functions. For more details about the permutations, see \cite{bona,Sta86}.

\bigskip

The modern study of permutation statistics began with the work of MacMahon \cite{Mac17}. These statistics have been extensively studied since the 70s by Foata and Schützenberger \cite{FS70,FZ90} and Françon \cite{Fra76}. The latter, together with Viennot \cite{FV79}, constructed a bijection linking the world of paths and permutations. The study of statistics has led to many results, such as the definition of Solomon's algebra of descents \cite{Sol76} thanks to the positions of the descents of a permutation.

\medskip

The most well-known permutation statistic is the number of inversion.  The concept of inversions was introduced by Cramer \cite{Cra} in 1750, and is used in the determinant formula of an $n\times n$ matrix as follows:
\[
\det\left(
  \begin{array}{ccc}
    a_{11} & \cdots & a_{1n} \\
    \vdots & \ddots & \vdots \\
    a_{a1} & \cdots & a_{nn} \\
  \end{array}\right) =\sum (-1)^{inv(\sigma_1 \cdots \sigma_n)}a_{1\sigma_1}\cdots a_{n\sigma_n},
\]
summed over all permutations $\sigma_1 \cdots \sigma_n$ of $\{1,\ldots,n\}$ where $inv(\sigma_1 \cdots \sigma_n)$
is the number of inversions of the permutation.

\medskip

As an analogue of the classical statistic inversion, Arslan et al. \cite{Ars2} introduced an inversion statistic on the hyperoctahedral
group $B_n$ by using a decomposition of a positive root system of this reflection group. This statistic allows Arslan to introduce in his paper \cite{Ars} the Mahonian numbers of type $B$.

\medskip

In our paper, we are interested in studying another combinatorial properties and interpretations of Mahonian numbers of type $B$, an analogue of these numbers with their combinatorial interpretations, then the log-concavity property of these numbers.

\medskip

The paper is structured as follows. In Sect. \ref{intro}, we give some definitions and notations that we need in this paper. In Sect. \ref{Cf}, we give the Knuth-Netto formula and  generating function for the subdiagonals on or below the main diagonal of the Mahonian number of type $B$, then its combinatorial interpretations by lattice path/partition and tiling. In Sect. \ref{AM}, we propose a $q$-analogue of Mahonian numbers of type $B$ by using a new statistics on the permutations of the hyperoctahedral group $B_n$ that we defined, then we study their basic properties and their combinatorial interpretations by lattice path/partition and tiling. In Sect. \ref{qlg}, we prove combinatorially that the $q$-analogue of Mahonian numbers of type $B$ form a strongly $q$-log-concave sequence of polynomials in $k$, which implies that the Mahonian numbers of type $B$ form a log-concave sequence in $k$ and therefore unimodal. Finally, in Sect. \ref{QR}, we propose a conjecture on the log-concavity of the sequence $\left( i_{B}(n,k)\right)_{n}$ (resp. the strong $q$-log-concavity of the sequence of polynomials $\left( i_{B_q}(n,k)\right)_{n}$)  in $n$, and a question about the number and location of modes of the unimodal sequence $(i_B(n,k))_{0\leq k \leq n^2}$.

\medskip

\section{Preliminaries and notations}
\label{intro}
Let $[n]:=\{1,\ldots,n\}$ and $\langle n \rangle:=\{-n,\ldots,-1,1,\ldots, n \}$. The {\bf symmetric group} $\mathcal{S}_n$  is the group of  permutations of $[n]$. For $\pi=\left(
  \begin{array}{ccc}
    1 & \cdots & n \\
    \pi_1 & \cdots & \pi_n \\
  \end{array}\right) \in \mathcal{S}_n$, the one-line presentation of $\pi$ is $\pi=\pi_1\cdots\pi_n$ to mean that $\pi_j=\pi(j)$ for all $j=1,\ldots,n$. The {\bf hyperoctahedral group} or {\bf signed symmetric group} $B_n$, which contains the symmetric group $\mathcal{S}_n$ as a subgroup, is the group of signed permutations of $\langle n \rangle$ such that $\pi(-j)=-\pi(j)$ for all $j\in [n]$. The hyperoctahedral groups are well studied (cf. \cite{r1, r2, r5, r6, r7, r11}). It is well known that $\mathcal{S}_n$ and $B_n$ are special cases of {\bf Coxeter groups}. For the combinatorics of the last groups, we refer the readers to the book of Björner and Brenti \cite{Bren}.

\bigskip

Now, we give some well-known statistics associated to the classical symmetric group.
\begin{definition}
Let $\sigma \in \mathcal{S}_n$.
\begin{itemize}
    \item An {\bf inversion} of $\sigma$ is a pair of indices $(i,j)$ such that $1\leq i <j\leq n$ and $\sigma(i) > \sigma(j)$. For the set of the inversions, we write
     \[Inv(\sigma)=\{ (i,j)\in [n]\times[n]: i<j \text{\ but \ }  \sigma(i)>\sigma(j)\},\]
    and $inv(\sigma)=|Inv(\sigma)|$ for the number of inversions of $\sigma$.
    \item A {\bf descent} of $\sigma$ is an index $j\in [n-1]$ such that $\sigma(j)>\sigma(j+1)$. For the set of descents, we write
    \[Des(\sigma)=\{ j\in [n-1]: \sigma(j)>\sigma(j+1)\},\]
    and $des(\sigma)=|Des(\sigma)|$ for the number of descents of $\sigma$.
    \item The {\bf major index} $maj(\sigma)$ of $\sigma$ is the sum of its descents:
    \[maj(\sigma)=\sum_{j\in Des(\sigma)}j.\]
\end{itemize}
\end{definition}

MacMahon proved in \cite{Mac} that the inversion statistic ”$inv$” is equi-distributed
with the major index ”$maj$” over the symmetric group $\mathcal{S}_n$, that is,
\begin{equation}\label{Im}
 \sum_{\sigma \in \mathcal{S}_n}q^{inv(\sigma)}=\sum_{\sigma \in\mathcal{S}_n}q^{maj(\sigma)}=\prod_{i=1}^{n}\frac{1-q^i}{1-q}:=[n]_q!  
\end{equation}
where $q$ is an indeterminate and $[n]_q!=[1]_q\cdots[n]_q$ is the $q$-analogue of $n!$.

\medskip

Let $I(n,k)=\{\sigma \in \mathcal{S}_n: inv(\sigma)=k\}$ for any two integers $n\geq 1$ and  $0\leq k\leq \binom{n}{2}$. Let $i(n,k)=| I(n,k)|$ the number of permutations of length $n$ having $k$ inversions, and is called as the {\bf Mahonian number} which was introduced for the first time by Rodrigues \cite{Red}. 

\bigskip

Yousra and Ahmia \cite{GA1} proved that $i(n,k)$ counts "the number of different ways of distributing "$k$" balls among "$n-1$"  boxes such that the $j$th box contains at most "$j$" balls".

\medskip

We have $\sigma=\left(
  \begin{array}{cccc}
    1 & 2 &\cdots & n \\
    n & n-1 & \cdots & 1 \\
  \end{array}\right)$ is the permutation with the maximum number of inversions in the group symmetric $\mathcal{S}_n$, i.e., $inv(\sigma)=\binom{n}{2}$. Then, we can express \eqref{Im} as follows:
\begin{equation}
\sum_{k=0}^{\binom{n}{2}}i(n,k)z^k=(1+z)\cdots (1+z+\cdots+z^{n-1}),
\end{equation}
where $i(n,k)=0$ unless $0\leq k\leq \binom{n}{2}$. For other properties and relations of $i(n,k)$, see \cite{Ind}. 

\medskip

The $q$-analogue of Mahonian numbers was first introduced by Ghemit and Ahmia in \cite{GA1}. They noted this $q$-analogue by $i_q (n,k)$, and they established \cite[Theorem 5]{GA1} that $i_q (n,k)$ satisfies the following generating function:
\begin{equation}
\sum_{k=0}^{\binom{n}{2}}i_q(n,k)z^k=\prod_{j=0}^{n-1}(1+q^jz+\cdots+(q^jz)^j)
\end{equation}
where $i_q(n,0)=1$, $i_q(n,k)=0$ unless $0\leq k\leq \binom{n}{2}$ and $i_{1}(n,k)=i(n,k)$.
Then, they provided \cite{GA1} its lattice path and tilling interpretations. Moreover, they investigated some combinatorial properties of $i_q(n,k)$ such as $q$-log-concavity property.

\medskip

Due to Ghemit and Ahmia \cite{GA1}, the $q$-Mahonian numbers have the following relations:

\begin{theorem}[Ghemit and Ahmia, \cite{GA1}]For $n>1$ and $0\leq k \leq \binom{n}{2}$, we have 
\begin{equation}\label{qrl1}
i_q(n,k)=\sum_{j=0}^{n-1}q^{j(n-1)}i_q(n-1,k-j),
\end{equation}   
\begin{equation}\label{qrl2}
i_q(n,k)=i_q(n-1,k)+q^{n-1}i_q(n-1,k-1)-q^{n(n-1)}i_q(n-1,k-n)
\end{equation}
and
\begin{equation}\label{qrl3}
i_q(n,k)=q^{\frac{n(n-1)(2n-1)}{6}}i_{\frac{1}{q}}\left(n,\binom{n}{2}-k\right).
\end{equation}
\end{theorem}

\bigskip

Return to the hyperoctahedral group $B_n$. From \cite{Bren},  this group has the canonical set of generators $S=\{t_1,s_1,\ldots,s_{n-1}\}$ and $t_{i+1}=s_i t_i s_i$ for each $i=1,\ldots,n-1$. It satisfies a semidirect product $B_n=\mathcal{S}_n\rtimes \mathcal{T}_n$, where $\mathcal{S}_n$ is generated by $\{s_1,\ldots,s_{n-1}\}$ and $\mathcal{T}_n$ is a normal subgroup of $B_n$ generated by $\{t_1,\ldots,t_{n}\}$. Then, it is clearly that $| B_n |=2^n n!$.

\medskip
Any signed permutation $\pi\in \mathcal{B}_n$  can be uniquely
written in the form
\[\pi=\left(
  \begin{array}{cccc}
    1 & 2 &\cdots & n \\
    (-1)^{r_1}\sigma_1&(-1)^{r_2}\sigma_2& \cdots & (-1)^{r_n}\sigma_n \\
  \end{array}\right)=\sigma\prod_{j=1}^n t_j^{r_j}\]
where $r_k=0$ or $1$ and $\sigma \in \mathcal{S}_n$.

\medskip

From \cite{Ad}, we can write also any signed permutation $\pi \in B_n$ as follows:
\begin{equation}\label{eqs}
\pi=\gamma_{n-1}^{k_{n-1}}\cdots\gamma_{1}^{k_{1}}\gamma_{0}^{k_{0}}    
\end{equation}
where $\gamma_{0}=t_1$, $\gamma_j=s_j\cdots s_1t_1$, $0\leq k_j\leq 2j+1$ for $j=1,\ldots, n-1$ and $k_0=0$ or $1$.

\bigskip

As new extension of the major index, Adin and Roichman \cite{Ad} defined red the following statistic for the hyperoctahedral group $B_n$ (the case $m=2$).
\begin{definition}Let $\pi \in B_n$.  {\bf The flag-major index} $fmaj(\pi)$ of $\pi$ is the following sum:
\[fmaj(\pi)=\sum_{j=0}^{n-1}k_j\]
where $k_j$, for $j=0,\ldots,n-1$, are the powers given in \eqref{eqs}.
\end{definition}
From \cite{Ad}, this statistic is also Mahonian:
\[\sum_{\pi \in B_n} q^{fmaj(\pi)}=\prod_{i=1}^{n}\frac{1-q^{2i}}{1-q}.\]

After this statistic, Arslan et al. \cite{Ars2} proposed an extension of inversion statistic for the hyperoctahedral group $B_n$ as follows (the case $m=2$ of Theorem 4.5 in \cite{Ars2}).
\begin{definition}\label{dib}LEt $\pi=\sigma\prod_{j=1}^n t_j^{r_j} \in B_n$. An {\bf inversion of type $B$} of $\pi$ is the sum of $i$-inversions $inv_i(\pi)$ of the permutation $\pi$, i.e.,
\[inv_B(\pi)=\sum_{i=1}^n inv_i(\pi)\]
where $inv_i(\pi)$ satisfies the following relation \cite[Theorem 4.5]{Ars2}:
\begin{equation}\label{eivi}
inv_i(\pi):=\begin{cases}
	inv_i(\sigma)=|\{(j,n+1-i): j<n+1-i, \sigma_j>\sigma_{n+1-i} \}|, &\text{if \ } r_{n+1-i}=0,\\
	1+2|\{(j,n+1-i): j<n+1-i, \sigma_j<\sigma_{n+1-i} \}|+inv_i(\sigma),&\text{if \ }r_{n+1-i}=1.
	\end{cases}    
\end{equation}
\end{definition}

\medskip

Define $Inv_B(\pi)=\{inv_1(\pi),\ldots,inv_n(\pi)\}$ to be the {\bf inversion table} of the permutation $\pi\in \mathcal{B}_n$. As an example, for $\pi=\left(
  \begin{array}{ccccccc}
    1 & 2 &3&4&5 & 6 \\
    5 & -1 & 2 & -4 & 6 & 3 \\
  \end{array}\right)\in \mathcal{B}_6$ the inversion table of $\pi$ is $Inv_B(\pi)=\{3,0,6,1,2,0\}$, and so $inv_B(\pi)=12$.

 \bigskip
We can easily write $B_n$ as follows: $$B_n=\biguplus_{j=-n}^{n}C_j$$ where $C_j=\{\pi \in B_n:\pi(n)=j\}$.

\medskip

Let $\pi=\left(
  \begin{array}{cccc}
    1 & \cdots &(n-1)& n \\
    \pi(1) & \cdots &\pi(n-1)& j\\
  \end{array}\right)\in C_j$ and $\tau$ be is a signed permutation defined by 
\[
\tau=\left(
  \begin{array}{ccc}
    a_1 & \cdots & a_{n-1} \\
    \pi(1) & \cdots & \pi(n-1) \\
  \end{array}\right)\in P\left([n]\backslash\{|j|\}\right)
\]
where $a_1,\ldots,a_{n-1}$ are an arrangement of elements of $[n]\backslash\{|j|\}$ in increasing order, and $P\left([n]\backslash\{|j|\}\right)$ is the group of all the signed permutation of the set $[n]\backslash\{|j|\}$. So, if we set 

\[
\pi_{\tau,j}=\left(
  \begin{array}{cccc}
    1 & \cdots &n-1& n\\
    \tau(a_1) & \cdots & \tau(a_{n-1})& j \\
   \end{array} \right)
\]
then we obtain $\pi=\pi_{\tau,j}$. Hence by \eqref{eivi}, we conclude that
\begin{equation}\label{rib}
inv_B(\pi):=\sum_{i=1}^n inv_i(\pi)=\begin{cases}
	n-j+inv_B(\tau), &\text{if \ } j>0,\\
	n-j-1+inv_B(\tau),&\text{if \ }j<0.
	\end{cases} 
\end{equation}

\medskip 
Arslan \cite{Ars} proved also that the total number of inversions of all permutations $\pi\in B_n$, namely $\mathcal{B}_n=\sum_{\pi \in B_n} inv_B(\pi)$, satisfies the following identities:
\begin{align}
&\mathcal{B}_n=2^{n-1}n^2n!,\label{eqtb0}\\
&\mathcal{B}_n=2^{n-1}n!(2n-1)+2n\mathcal{}{B}_{n-1}  \text{\ \  for\ } n\geq 2,\label{eqtb}
\end{align}
with $\mathcal{B}_1=1$.

\bigskip

Arslan et al. \cite{Ars2} established that the inversion statistic of type $B$ ”$inv_B$” is equi-distributed with the
flag-major index ”$fmaj$” over the hyperoctahedral group $B_n$, that is,
\begin{equation}\label{fIm}
 \sum_{\pi \in B_n}q^{inv_B(\pi)}=\sum_{\sigma \in B_n}q^{fmaj(\pi)}=\prod_{i=1}^{n}\frac{1-q^{2i}}{1-q}.  
\end{equation}

\medskip
The permutation of $B_n$ having maximum number of inversions of type $B$ is
$$\pi_0=\left(
  \begin{array}{ccccc}
    1 & 2 &\cdots &n-1& n \\
    -1 & -2 & \cdot &-(n-1)& -n \\
  \end{array}\right),$$
which has $inv_B(\pi_0)=n^2$.

\bigskip

The previous notations and results of the hyperoctahedral group $B_n$, allow to Arslan in \cite{Ars} to introduce the number of signed permutations of length $n$ with exactly $k$ inversions, which is denoted by $i_B(n, k)$ and called {\bf Mahonian numbers of type $B$}, such that $$i_B(n, k):=|I_B(n,k)|=|\{\pi \in B_n: inv_B(\pi)=k\}|$$
and
\begin{equation}
 \sum_{k=0}^{n^2}i_B(n, k)q^k=(1+q)(1+q+q^2+q^3)\cdots(1+q+\cdots+q^{2n-1}).   
\end{equation}

\medskip

Similar to the combinatorial interpretation established by Ghemit and Ahmia \cite{GA1} for  classical Mahonian numbers $i(n,k)$, Arslan \cite{Ars} gave the following combinatorial interpretation for Mahonian numbers of type B:

\medskip

\noindent{\bf Combinatorial interpretation:} $i_B(n, k)$ counts the number of ways to place
”$k$” balls into ”$n$” boxes such that the $j$th box contains at most ”$2j-1$” balls.

\bigskip

He called \cite{Ars} the table values of $i_B(n, k)$ as the {\bf Mahonian triangle of type $B$}. This table also appears in Sloane \cite{slo} as {\bf A128084}. See Table 1.

\begin{table}[h!]
		
		\begin{center}
\setlength{\tabcolsep}{4pt}

	\begin{tabular}{|c|c|ccccccccccccc|}
 \hline
	$n/k$ & $\mathcal{B}_n$& \ \ 0 \ \ & \ \ 1 \ \ & \ \ 2 \ \ & \ \ 3 \ \ & \ \ 4 \ \ & \ \ 5 \ \ & \ \ 6 \ \ & \ \ 7 \ \ & \ \ 8 \ \ & \ \ 9 \ \ & \ \ 10 \ \ & \ \ $\cdots$ \ \ &    \\
 \hline
    0 & {\bf 0} &1 &  &  &  &  &  &  &  &  &  &  &  &      \\
	1 & {\bf 1} &1 & 1 &  &  &  &  &  &  &  &  &  &  &      \\
	2 & {\bf 16} & 1 & 2 & 2 & 2 & 1 &  &  &  &  &  &  &  &      \\
	3 & {\bf 216} & 1 & 3 & 5 & 7 & 8 & 8 & 7 & 5 & 3 & 1 &  &    &    \\
	4 & {\bf 3072} & 1 & 4 & 9 & 16 & 24 & 32 & 39 & 44 & 46 & 44 & 39 & $\cdots$  &   \\  
 \hline
	\end{tabular}

		\caption{The Mahonian triangle of type $B$.}\label{t}
	\end{center}
		\end{table}
  
\medskip

For $n\geq 1$ and $0\leq k\leq n^2$, the Mahonian number of type B satisfies the following recurrence relations \cite[Theorems 4.2 and 4.3]{Ars}:
\begin{align}
&i_B(n,k)=i_B(n,n^2-k), \label{eb1} \\   
&i_B(n,k)=\sum_{j=0}^{2n-1}i_B(n-1,k-j) \text{\ for\ }n\geq2,\label{eb2}
\end{align}
where $i_B(1,0)=i_B(1,1)=1$ and $i_B(n,k)=0$ unless $0\leq k\leq n^2$.

\bigskip

\noindent According to all properties of $i_B(n,k)$ and $\mathcal{B}_n$, Arslan \cite{Ars} concluded that $\mathcal{B}_n=\sum_{k=0}^{n^2}i_B(n,k)k$.

\section{Combinatorial formulas and interpretations}\label{Cf}

This section is devoted to some combinatorial formulas: a recurrence relation, the Knuth-Netto formula, and generating function for the subdiagonals on or below the main diagonal of Mahonian numbers of type $B$ and its combinatorial interpretations using lattice paths and partitions/tilings. The lattice path interpretation allows us  to say that there is a bijection between these lattice paths and the signed permutations of the hyperoctahedral group $B_n$, and this means that we can build these permutations through these lattice paths.

\bigskip

We can easily obtain from \eqref{eb2} that the following recurrence relation for the Mahonian numbers of type $B$ holds.
\begin{proposition}
For $n>1$ and $0\leq k\leq n^2$, we have
\begin{equation}\label{recb}
i_B(n,k)=i_B(n,k-1)+i_B(n-1,k)-i_B(n-1,k-2n). 
\end{equation}
\end{proposition}
\subsection{The Knuth-Netto formula of Mahonian numbers of type $B$}
In this subsection, we give the analogue of the Knuth-Netto \cite{Kn,net} formula for the case of Mahonian numbers $i_B(n,k)$ of type $B$, a question proposed by Arslan in his paper \cite{Ars}. 

\medskip

First of all, we give Knuth-Netto \cite{Kn,net} formula for the $k$th Mahonian number $i(n, k)$ when $k\leq n$ as follows,
\begin{equation}\label{KN}
i(n,k)=\binom{n+k-1}{k}+\sum_{j=1}^{+\infty}(-1)^j \binom{n+k-u_j -j-1}{k-u_j-j}+\sum_{j=1}^{+\infty}(-1)^j \binom{n+k-u_j -1}{k-u_j},
\end{equation}
where $u_j=\frac{j(3j-1)}{2}$ is the $j$th pentagonal number, see Figure \ref{ff}. 

\begin{figure}[h!]
   \centering
   \includegraphics[width=5cm]{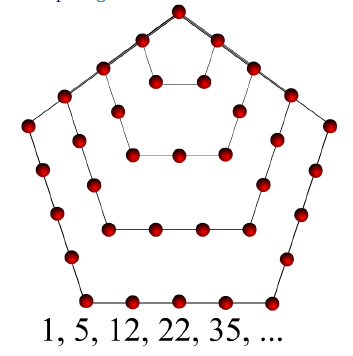}
   \caption{The pentagonal numbers.}\label{ff}
 \end{figure}

\medskip

This formula follows from the generating function and {\bf Euler}'s pentagonal number theorem.
\begin{theorem}\cite{And1,Har,Ind} We have,
\begin{equation}\label{Euler}
\prod_{j=1}^{+\infty}(1-q^j)=\sum_{k=-\infty}^{+\infty}(-1)^kq^{u_k}=1+\sum_{k=1}^{+\infty}(-1)^k\left(q^{\frac{k(3k-1)}{2}}+q^{\frac{k(3k+1)}{2}} \right).
\end{equation}
\end{theorem}

By letting $q\mapsto q^2$ in the previous theorem, we find an analogue of the generating function and Euler's pentagonal number theorem as follows.
\begin{corollary}\label{cre} We have,
\begin{equation}\label{qeuler}
\prod_{j=1}^{+\infty}(1-q^{2j})=\sum_{k=-\infty}^{+\infty}(-1)^kq^{2u_k}=1+\sum_{k=1}^{+\infty}(-1)^k\left(q^{k(3k-1)}+q^{k(3k+1)} \right).
\end{equation}
\end{corollary}
Using this corollary, we obtain the analogue of the Knuth-Netto  formula associated to Mahonian numbers $i_B(n,k)$ of type $B$.
\begin{theorem}\label{knuth-netto}For $n\geq1$ and $0 \leq k\leq n$, we have 
\begin{align}
i_B(n,k)=\binom{n+k-1}{k}+\sum_{j=1}^{+\infty}(-1)^j  \binom{n+k-2u_j -2j-1}{k-2u_j-2j}+\sum_{j=1}^{+\infty}(-1)^j \binom{n+k-2u_j -1}{k-2u_j}.&
\end{align} 
\end{theorem}
\begin{proof}From the generating function of $i_B(n,k)$, we have
\begin{align*}
\sum_{k=0}^{n}i_B(n,k)q^k&=\prod_{j=1}^{n}\frac{1-q^{2j}}{1-q}=\frac{1}{(1-q^n)}\prod_{j=1}^{n}(1-q^{2j})\\
&=\prod_{j=1}^{n}(1-q^{2j})\sum_{l=0}^{+\infty}\binom{n+l-1}{l}q^l.
\end{align*}
The coefficients of $\prod_{j=1}^{n}(1-q^{2j})$ in the last equation will match those in the power series expansion of the
infinite product of Corollary \ref{cre} given by the analogue of Euler's pentagonal number theorem up to the coefficient on $q^n$. 

\medskip

By Corollary \ref{cre}, we consider the product
\begin{align*}
\prod_{j=1}^{+\infty}(1-q^{2j})\sum_{l=0}^{+\infty}\binom{n+l-1}{l}q^l=\left(1+\sum_{j=1}^{+\infty}(-1)^j\left(q^{j(3j-1)}+q^{j(3j+1)} \right)\right)\sum_{l=0}^{+\infty}\binom{n+l-1}{l}q^l.&\\
\end{align*}
From this equation, the coefficient of $q^k$, for $k \leq n$, is 
\begin{eqnarray*}
\binom{n+k-1}{k}+\sum_{j=1}^{+\infty}(-1)^j \binom{n+k-2u_j -2j-1}{k-2u_j-2j}+\sum_{j=1}^{+\infty}(-1)^j \binom{n+k-2u_j -1}{k-2u_j},&&
\end{eqnarray*}
which is exactly the right side of the desired identity of $i_B(n,k)$.
\end{proof}

\begin{example}
We can illustrate the formula in Theorem \ref{knuth-netto} as follows:
$$i_B(4,7)=\binom{9}{5}-\binom{8-2u_1}{5-2u_1}-\binom{10-2u_{1}}{7-2u_{1}}=44$$
where $u_1=1$.
\end{example}

\subsection{The generating functions for the subdiagonals of Mahonian triangle of type $B$}
In this subsection, we shall derive generating functions for the subdiagonals on or below the main diagonal
of {\bf the Mahonian triangle of type $B$}. For $j\geq 0$, let $T_j(x)=\sum_{n\geq 0}i_B(n,n-j)x^n$ is the generating function for signed permutations of $\langle n\rangle =\{-n,\ldots,-1,1,\ldots,n\}$, of the hyperoctahedral group $\mathcal{B}_n$, with exactly $n-j$ inversions. 
\begin{example}
The first four functions of $T_j(x)$ are
\begin{align*}
&T_0(x)=1+x+2x^2+7x^3+24x^4+86x^5+\cdots \\
&T_1(x)=x+2x^2+5x^3+16x^4+54x^5+190x^6+\cdots \\
&T_2(x)=x^2+3x^3+9x^4+30x^5+104x^6+371x^7+\cdots\\
&T_3(x)=x^3+4x^4+14x^5+50x^6+181x^7+664x^8+\cdots
\end{align*}
\end{example}

By using the concepts of inversion table in the symmetric group $\mathcal{S}_n$ and subdiagonal sequence, Calaesson et al. \cite{CFS} showed that the generating function for permutations of $[n]$ with exactly $n-j$ inversions, denoted by $S_j(x)$, satisfies the following relation:
\begin{equation}\label{gfp}
S_j(x)=\left(xC(x)\right)^j S_0(x)
\end{equation} 
where $S_0(x)=1+x^3+5x^4+22x^5+\cdots$ is the generating function for permutations of $[n]$ with exactly $n$ inversions, and $C(x)=\frac{1-\sqrt{1-4x}}{2x}$ is the generating function of the Catalan numbers $C_n=\frac{1}{n+1}\binom{2n}{n}$, which are shown \cite[Lemma 1]{CFS} equal to number of weakly increasing subdiagonal sequences of length $n$. 

\medskip

By the same concepts used by Calaesson et al. in \cite{CFS} but in the hyperoctahedral group $\mathcal{B}_n$, we establish in the following theorem a similar relation to \eqref{gfp} for the generating function of signed permutations of $\langle n\rangle$ with exactly $n-j$ inversions.

\medskip

We start by the following notation: Let $\mathcal{C}_B(n)$ be the subset of $I_B(n,n-1)$ (i.e., the set  of permutations of $B_n$ having $n-1$ inversions of type $B$) consisting of those permutations whose every prefix of length $k\geq 1$ has fewer than $k$ inversions of type $B$. 

\medskip

Calaesson et {\it al.} \cite[Lemma 1]{CFS} proved that $|\mathcal{C}_n|=C_{n-1}$, where  $\mathcal{C}_n$ is the subset of the set $I(n,n-1)$ (i.e., the set of classical permutations of $n$ having $n-1$
inversions) consisting of those permutations whose every prefix of length $k\geq1$
has fewer than $k$ inversions.

\medskip

From the definition of $\mathcal{C}_B(n)$ and \cite[Lemma 1]{CFS}, we can give the following lemma.
\begin{lemma}\label{ctln}
For $n\geq 1$, $| \mathcal{C}_B(n)|=|\mathcal{C}_n|=C_{n-1}$.
\end{lemma}
\begin{proof}
To prove this lemma, it suffices to show that the set $\mathcal{C}_B(n)$ is exactly the set $\mathcal{C}_n$. 

\medskip

The subset of $I_B(n,n-1)$ consisting of those permutations whose every $k$-prefix, $k\geq 1$,
has fewer than $k$ inversions of type $B$ is the  same $\mathcal{C}_n$, because the only permutations of $B_n$ that have $k$-prefixes have fewer than $k$ inversions of type $B$ are, from Definition \ref{dib} and relation \eqref{eivi}, those that have elements with a positive sign. This gives us the same permutations of the set $\mathcal{C}_n$. Consequently, the sets $\mathcal{C}_B(n)$ and $\mathcal{C}_n$ are the same, which gives us the result of the lemma, see Example \ref{exd}.  
\end{proof}
\begin{example} \label{exd}
For $n=3$, the set of signed permutations of length $3$ having $2$ inversions is $$I_B(3,2)=\{231,312,2-13,-132,-213\}.$$ 
Then, the set of permutation of length $3$ whose each prefix of length $k\geq 1$ has fewer than $k$ inversions of type $B$ is $$\mathcal{C}_B(3)=\{231,312\}.$$
Thus, $|\mathcal{C}_B(3)|=C_2=2$. 
\end{example}

Recall that $T_j(x)$ is the generating function for signed permutations of length $n$ with $n-j$
inversions of type $B$:
\[
T_j(x)=\sum_{n\geq 0} | I_B(n,n-j) |x^n.
\]
We omit the proof of Theorem \ref{thmc} for brevity since it can be easily seen using Lemma \ref{ctln} and the same arguments in \cite[Theorem 3]{CFS}.
\begin{theorem}\label{thmc}For $j\geq 1$, we have
\[
I_B(n,n-j-1)\simeq \bigsqcup_{i=0}^{n}I_B(i,i-j)\times \mathcal{C}_B(n-i)
\]
and thus the generating function $T_{j+1}(x)$ and $T_{j}(x)$ satisfy the identity
$$T_{j+1}(x)=xC(x)T_{j}(x),$$
equivalently
\begin{equation}\label{gfb}
T_j(x)=\left(xC(x)\right)^j T_0(x).
\end{equation}
\end{theorem}

Calaesson et al. \cite[Theorem 4]{CFS} used the partition theory to prove that $S_0(x)$ satisfies the following relation:
\begin{equation}\label{gfs}
S_0(x)=R\left(xC(x)\right),
\end{equation} 
where $R(x)=\frac{1-x}{1-2x}\prod_{n=1}^{+\infty}(1-x^n)$.

\bigskip

By the same approach given in \cite[Theorem 4]{CFS}, we establish the following theorem which generates the coefficients of the generating function $T_0(x)$. We omit the details of the proof for brevity.
\begin{theorem}We have,
\begin{equation}\label{gfbb}
T_0(x)=L\left(xC(x)\right),
\end{equation} 
where $L(x)=\frac{1-x}{1-2x}\prod_{n=1}^{+\infty}(1-x^{2n})$.
\end{theorem}

\subsection{Lattice path interpretation}
Due to Ghemit and Ahmia \cite{GA1}, the Mahonian number $ i(n,k) $ counts the number of lattice paths from $ u=(0,0) $ to $ v=(n-1,k) $ taking at most $j$ North steps at the level $j$. This interpretation allows us to give in this subsection a combinatorial interpretation by lattice paths for the Mahonian numbers of type B.
Let $\mathcal{P}^{B}_{n,k} $ denote the set of lattice paths from the point $u=(0,0) $ to the point $v=(n,k) $  for $n\geq 1$ and $0\leq k \leq n^2$, with only North steps (vertical steps $=(0,1)$) and East steps (horizontal steps $=(1,0)$), such that the number of North steps in each level $j\geq 1$ is at most $2j-1$, where the levels associated to vertical lines are from $1$ to $n$, as shown in the example of Figure \ref{fgf}.

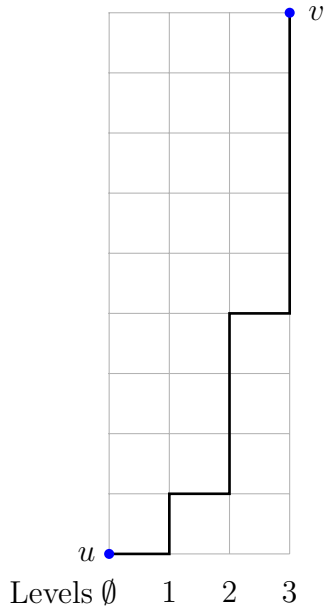
\begin{figure}[h!]
\begin{center}
\begin{tikzpicture}
\draw[step=0.8cm,color=black!30] (-0.81,0) grid (1.6,7.2);
\draw [line width=1pt](-0.8,0) -- (0,0) -- (0,0.8) -- (0.8,0.8) -- (0.8,3.2)--(1.6,3.2)--(1.6,4)--(1.6,4.8)--(1.6,5.6)--(1.6,6.4)--(1.6,7.2);
\fill[blue] (-0.8,0) circle(1.9pt) ;
\fill[blue] (1.6,7.2) circle(1.9pt) ;
\node[right=0.1pt] at (1.7,7.2){$v$};
\node at (-1.1,0){$u$};
\node at (-1.5,-0.5){\text{Levels\ }};
\node at (-0.8,-0.5){$\emptyset$};
\node at (0,-0.5){$1$};
\node at (0.8,-0.5){$2$};
\node at (1.6,-0.5){$3$};
\end{tikzpicture}
\end{center}
\caption{A path $P$ in $\mathcal{P}^{B}_{3,9}$.}
  \label{fgf}
\end{figure}

\medskip

Using Definition \ref{dib} and the previous notations, we can interpret the Mahonian numbers of type $B$ as follows.
\begin{theorem}\label{thib}
The Mahonian numbers of type B  counts the number of lattice paths from $ u=(0,0) $ to $ v=(n,k) $ taking at most $ 2j-1 $ North steps at the level $ j $ for $j\geq1$, that is,
$$i_{B}(n,k)= \mid \mathcal{P}^{B}_{n,k} \mid.$$
 \end{theorem}
\begin{proof}
To prove this theorem, it suffices to show that there is a bijection between the set of signed permutations of the hyperoctahedral
group $B_n$ having $k$ inversions of type $B$ (i.e., the inversions that satisfy the relations of Definition \ref{dib}) and  the set of lattice paths $\mathcal{P}^{B}_{n,k}$. So we proceed as follows:

\smallskip

\noindent For each path $P\in \mathcal{P}^{B}_{n,k}$, we can easily find the signed permutation of length $n$ having $k$ inversions of type $B$ associated to $P$, this permutation is obtained as follows: we associate to the point $(0,0)$ the entry $\emptyset$, the first step is necessarily an East step, we move to the point $(1,0)$ and we associate to this point the entry $1$. \textbf{At the point (1,0)}, we have two cases: \textbf{if} the next step of $P$ is an East step, we add the entry $2$ to the right of $1$ (i.e., $12$), \textbf{and if} the next step of $P$ is an North step, here we add the negative sign to the entry $1$ (i.e.,$-1$, which gives us an inversion of type $B$ from Definition \ref{dib}). \textbf{At the point (1,1)}, here, the next step of $ P $ is necessarily an East step (since at level $j=1$, the maximum number of North steps is $2j-1=1$), we add the entry $ 2 $ to the right of $ -1 $ (i.e., $-12$). We move to \textbf{the point (2,1)}, here we have two cases :   \textbf{if} the next step of $ P $ is an East step, we add the entry $3$ to the right of $ -12 $ (i.e, $-123$), \textbf{else if} the next step of $ P $ is an North step, here we permute -1 with 2 and we permute the signs (i.e. $-21$, which gives us an inversion of type $B$).

\medskip
 
In the general case, if at a given point we have the signed permutation $\pi=\pi_{1}\cdots\pi_{l}$. \textbf{If} the next step of $P$ is an East step, we add the entry $(l+1)$ to the right of $\pi$ (i.e., $\pi_{1}\cdots\pi_{l}(l+1)$), \textbf{if} the next step of $P$ is an North step, here we have three cases : \textbf{if $\pi_{l}=1$} we add the negative sign to $\pi_{l}$ (i.e., $\pi_{1}\cdots\pi_{l-1}(-1)$, which gives us an inversion of type $B$), \textbf{if $\pi_{l}>0$} we search $\pi_{m}$ such that $\pi_{l}=\mid\pi_{m}\mid+1$, here if $\pi_{m}>0$, we permute $\pi_{m}$ with $\pi_{l}$ (i.e., $\pi_{1}\cdots\pi_{l}\cdots\pi_{m}$, which gives us an inversion of type $B$), and if $\pi_{m}<0$, we permute $\pi_{m}$ with $\pi_{l}$ and we permute their signs (i.e., $\pi_{1}\cdots-\pi_{l}\cdots-\pi_{m}$, which gives us an inversion of type $B$), \textbf{if $\pi_{l}<0$} we search $\pi_{m}$  such that $\mid\pi_{l}\mid=\mid\pi_{m}\mid-1$, here if $\pi_{m}<0$, we permute $\pi_{m}$ with $\pi_{l}$ (i.e., $\pi_{1}\cdots\pi_{l}\cdots\pi_{m}$, which gives us an inversion of type $B$), and if $\pi_{m}>0$, we permute $\pi_{m}$ with $\pi_{l}$ and we permute their signs (i.e., $\pi_{1}\cdots-\pi_{l}\cdots-\pi_{m}$, which gives us an inversion of type $B$). We proceed the same operations with next points until we arrive to the point $(n,k)$, that gives us the desired permutation having $k$ inversions of type $B$. An example is shown in Figure \ref{fg3}.
\end{proof}

\begin{figure}[h!]
\begin{center}
\begin{tikzpicture}
\draw[step=0.8cm,color=black!30] (-0.81,0) grid (1.6,7.2);
\draw [line width=1pt](-0.8,0) -- (0,0) -- (0,0.8) -- (0.8,0.8) -- (0.8,3.2)--(1.6,3.2)--(1.6,4)--(1.6,4.8)--(1.6,5.6)--(1.6,6.4)--(1.6,7.2);

\fill[blue] (-0.8,0) circle(1.9pt) ;
\fill[blue] (0,0) circle(1.9pt) ;
\fill[blue] (0,0.8) circle(1.9pt) ;
\fill[blue] (0.8,0.8) circle(1.9pt) ;
\fill[blue] (0.8,1.6) circle(1.9pt) ;
\fill[blue] (0.8,2.4) circle(1.9pt) ;
\fill[blue] (0.8,3.2) circle(1.9pt) ;
\fill[blue] (1.6,3.2) circle(1.9pt) ;
\fill[blue] (1.6,4) circle(1.9pt) ;
\fill[blue] (1.6,4.8) circle(1.9pt) ;
\fill[blue] (1.6,5.6) circle(1.9pt) ;
\fill[blue] (1.6,6.4) circle(1.9pt) ;
\fill[blue] (1.6,7.2) circle(1.9pt) ;

\node[blue,rectangle] at (-0.8,-0.4) {$\emptyset$};
\node[blue,rectangle] at (0.0,-0.4) {$1$};
\node[blue,rectangle] at (-0.4,0.9) {$-1$};
\node[blue,rectangle] at (1.2,0.9) {$-12$};
\node[blue,rectangle] at (1.2,1.7) {$-21$};
\node[blue,rectangle] at (1.4,2.5) {$-2-1$};
\node[blue,rectangle] at (0,3.3) {$-1-2$};
\node[blue,rectangle] at (2.4,3.3) {$-1-23$};
\node[blue,rectangle] at (2.4,4.1) {$-1-32$};
\node[blue,rectangle] at (2.4,4.9) {$-2-31$};
\node[blue,rectangle] at (2.7,5.7) {$-2-3-1$};
\node[blue,rectangle] at (2.7,6.5) {$-1-3-2$};
\node[blue,rectangle] at (2.7,7.3) {$-1-2-3$};
\node[right=0.1pt] at (1.1,7.4){$v$};
\node at (-1.1,0){$u$};
\end{tikzpicture}
\end{center}
	\caption{The path of $\pi =-1-2-3$ which has $9$ inversions of type $B$.}
	\label{fg3}

\end{figure}
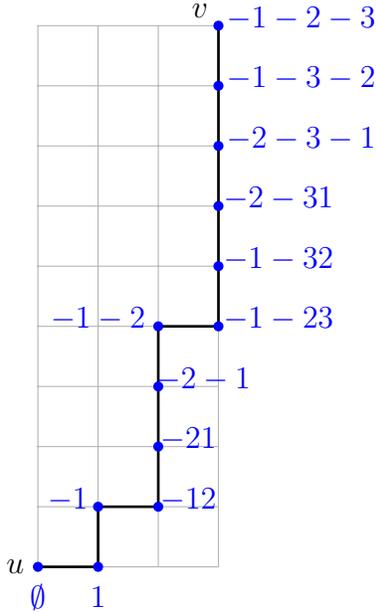 

\subsection{Partition/tiling interpretation}
The integer partitions \cite{And3, And, And1,And2} plays important roles in combinatorics, number theory and other
related mathematical branches.

\medskip
In this subsection, we give a partition/tiling interpretation for the Mahonian numbers of type $B$. For the  partition/tiling interpretation of the classical Mahonian number, we refer the readers  to the paper of Ghemit and Ahmia \cite{GA1}.

\medskip
Let $n$ and $k$ denote two nonnegative integers. A partition $\lambda$ of size $n$ and length $k$ is an $k$-tuple $\lambda=(\lambda_1,\lambda_2,\ldots,\lambda_k)$ of integers such that
$ \lambda_{1}\geq \lambda_{2}\geq \cdots \geq \lambda_k\geq1$, $ \lambda_{1}+ \lambda_{2}+ \cdots +\lambda_k=n$ and each $\lambda_{j}$ is a part of  $\lambda$. We shall refer by $l(\lambda)$ to the number of parts of the partition $\lambda$, and the multiplicity $m_{j}= m_{j}(\lambda)$ of part $j$ in  $\lambda$ is is the number of occurrences of $j$ as a part in $\lambda$. For example, $5$ can be partitioned in seven distinct ways: $5,4+1, 3+2, 3+1+1, 2+2+1 , 2+1+1+1,1+1+1+1+1$. Notice that the unique partition of $0$ is the empty partition $\emptyset$, which has length $0$. The partition $\lambda=(2,2,1)$ has a
number of parts $l(\lambda) = 3$ and can represent it geometrically by a Young diagram, see Figure \ref{fc}.

\begin{figure}[h]
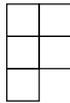

\begin{center}
\ytableausetup
{mathmode, boxframe=normal, boxsize=1em}
\begin{ytableau}

 &   \\
 &   \\
   \\
\end{ytableau}\\
\caption{The partition $\lambda=(2,2,1)$.}\label{fc}
\end{center}
\end{figure}

Recall to the path interpretation of Mahonian numbers of type $B$ (i.e., Theorem \ref{thib}), $i_B(n,k)$ counts the number of lattice paths from
$(0,0)$ to $(n,k)$ taking at most $2j-1$ North steps at the level $j$ for $j\geq 1$. So, we can easily see the correspondence between North-East paths and partitions, just see
that each number of boxes above and to the left of the path. 

\medskip

Let $\mathcal{P}r_{n,k}$ the set of partitions into $k$ parts, and the largest part is less or equal to $n$, such that each part $j$ can be repeated at most $2j-1$ times. Then, we obtain the following partition interpretation of Mahonian numbers of type $B$. 
\begin{theorem}\label{thib2}
For $n\geq 1$ and $0\leq k \leq n^2$, the Mahonian number of type $B$ counts the number of partitions into $k$ parts, and the largest part is less or equal to $n$, such that each part $j$ can be repeated at most $2j-1$ times, that is,
\[i_B(n,k)=\mid \mathcal{P}r_{n,k}\mid.\]
\end{theorem}
For example  in Figure \ref{fgr}, each row represents a part, and then,
these parts give us a partition. 

\begin{figure}[h!]
	
	\begin{center}
		\begin{tikzpicture}

		\draw[step=0.8cm,color=black!30] (-0.81,0) grid (1.6,2.4);
		\draw [line width=1pt](-0.8,0) -- (0,0) -- (0,0.8) -- (0.8,0.8) -- (0.8,2.4)--(1.6,2.4);
		
		\fill[blue] (-0.8,0) circle(1.9pt) ;
		\fill[blue] (1.6,2.4) circle(1.9pt) ;
		\node[right=0.1pt] at (1.7,2.4){$v$};
		\node at (-1.1,0){$u$};

		\fill[fill=gray!50,draw=black!50] (-0.75,0.05) rectangle (-0.05,0.75);	\fill[fill=gray!50,draw=black!50] (-0.75,0.85) rectangle (-0.05,1.55);
		\fill[fill=gray!50,draw=black!50] (0.05,0.85) rectangle (0.75,1.55);
		\fill[fill=gray!50,draw=black!50] (-0.75,1.65) rectangle (-0.05,2.35);
		\fill[fill=gray!50,draw=black!50] (0.05,1.65) rectangle (0.75,2.35);	
		\end{tikzpicture}
	\end{center}
	
	\caption{The path corresponds to the partition $\lambda=(2,2,1)$.}
	 \label{fgr}
\end{figure}

\medskip

To establish the tiling interpretation of Mahonian numbers of type $B$, we give the following definition: let $\mathcal{T}^{B}_{n,k}$ be the set of weighted tilings of an $(n+k)\times 1$-board in which we use only $n$ green squares and $k$ orange squares where the number of successive  orange squares  is at most $2j-1$ if there are $j$ green squares before. Thus, we have the following tiling interpretation.
\begin{theorem}\label{thib3}
	The $i_B(n,k)$ counts the number of tilings of size $(n+k)\times 1$-board taking only $n$ green squares and $k$ orange squares,
where the number of successive 
orange squares is at most $2j-1$ if there are $j$ green squares before, that is,
\[i_B(n,k)=\mid \mathcal{T}^{B}_{n,k} \mid.\]
\end{theorem}
\begin{proof}
The set of tilings $\mathcal{T}^{B}_{n,k}$ is in bijection with set of paths $\mathcal{P}^{B}_{n,k}$. Each green square is corresponded to an East step, and each orange square to a North step, and vice versa. Moreover, this bijection is weight-preserving.
\end{proof}

As an example, in Figure \ref{tqpq} we have a tiling $T\in \mathcal{T}^{B}_{3,3}$  and its corresponded path.
\begin{figure}[h!]
	
	\begin{center}
		\begin{tikzpicture}
		
		\draw[step=0.8cm,color=black!30] (0,0) grid(4.8,0.8);
		\fill[fill=green!60,draw=black!50] (0.05,0.05) rectangle (0.75,0.75);
		\fill[fill=orange!90,draw=black!50] (0.85,0.05) rectangle (1.55,0.75);
		\fill[fill=green!60,draw=black!50] (1.65,0.05) rectangle (2.35,0.75);
		\fill[fill=orange!90,draw=black!50] (2.45,0.05) rectangle (3.15,0.75);
		\fill[fill=orange!90,draw=black!50] (3.25,0.05) rectangle (3.95,0.75);
		\fill[fill=green!60,draw=black!50] (4.05,0.05) rectangle (4.75,0.75);

		\end{tikzpicture}\\
		\vspace{0.1cm}
		\begin{tikzpicture}

		\draw[step=0.8cm,color=black!30] (-0.81,0) grid (1.6,2.4);
		\draw [line width=1pt](-0.8,0) -- (0,0) -- (0,0.8) -- (0.8,0.8) -- (0.8,2.4)--(1.6,2.4);
		
		\fill[blue] (-0.8,0) circle(1.9pt) ;
		\fill[blue] (1.6,2.4) circle(1.9pt) ;
		\node[right=0.1pt] at (1.7,2.4){$v$};
		\node at (-1.1,0){$u$};
		
		\fill[fill=gray!50,draw=black!50] (-0.75,0.05) rectangle (-0.05,0.75);	\fill[fill=gray!50,draw=black!50] (-0.75,0.85) rectangle (-0.05,1.55);
		\fill[fill=gray!50,draw=black!50] (0.05,0.85) rectangle (0.75,1.55);
		\fill[fill=gray!50,draw=black!50] (-0.75,1.65) rectangle (-0.05,2.35);
		\fill[fill=gray!50,draw=black!50] (0.05,1.65) rectangle (0.75,2.35);

		\end{tikzpicture}
		
	\end{center}
	\caption{A tiling  $T$ and its corresponded path.}\label{tqpq}
\end{figure}
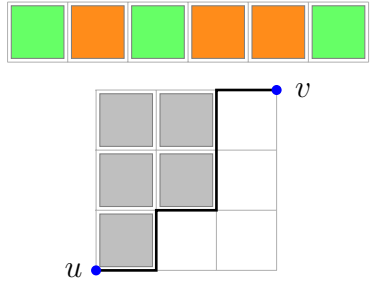
\section{$q$-Analogue of Mahonian numbers of type $B$}\label{AM}
In this section, we propose a $q$-analogue of Mahonian numbers of type $B$ and we study
some basic properties. To do this, we use a new statistics on signed permutations of hyperoctahedral group $B_n$, this statistics based on inversion  $inv_i(\pi)$ defined by Arslan \cite{Ars}. Besides that, we give some combinatorial interpretations
by lattice paths/partitions and tilings for this analogue of Mahonian numbers of type $B$.
\subsection{Definition and some identities}
Let me now give the following statistic: For each signed permutation $\pi \in B_n$, let
\begin{equation}\label{statt}
\omega(\pi)=\sum_{i=1}^n (n-i+1)inv_i(\pi).   
\end{equation}

Using this statistic, we propose in the following definition  a $q$-analogue of $i_B(n, k)$.
\begin{definition}\label{dfq}
For $n\geq 1$ and $0\leq k \leq n^2$, we define the $q$-Mahonian number of type $B$ as follows:
\begin{equation}\label{eqan}
i_{B_q}(n,k)=\sum_{\pi \in I_B(n,k)}q^{\omega(\pi)},
\end{equation}
where $I_B(n,k)=\{\pi \in B_n: inv_B(\pi)=k\}$.
\end{definition}

If $\pi \in C_j$, for $j\in\{-n,\ldots,-1,1,\ldots,n\}$, we can rewrite \eqref{statt} from \eqref{rib} as follows:
\begin{equation}\label{wrib}
\omega(\pi)=\omega(\pi_{\tau,j})=\begin{cases}
	n(n-j)+\omega(\tau), &\text{if \ } j>0,\\
	n(n-j-1)+\omega(\tau),&\text{if \ }j<0,
	\end{cases} 
\end{equation}
where $\tau \in P\left([n]\backslash\{|j|\}\right)$.

\medskip

By Definition \ref{dfq} and relation \eqref{wrib}, we can establish that the $q$-Mahonian numbers of type $B$ satisfy the following recurrence relation which is the $q$-analogue of the relation \eqref{eb2}.
\begin{theorem}\label{qrl}
For $n>1$, we have
\begin{equation}\label{rqr}
i_{B_q}(n,k)=\sum_{j=0}^{2n-1}q^{nj}i_{B_q}(n-1,k-j),   
\end{equation}
where $i_{B_q}(1,0)=1$ and $i_{B_q}(1,1)=q$.
\end{theorem}
\begin{proof}It is clear  that $i_{B_q}(1,0)=1$ and $i_{B_q}(1,1)=q$. Now from \eqref{statt} and Definition \ref{dfq}, for $n\geq 2$, we have
\begin{eqnarray*}
i_{B_{q}}(n,k) &=&\sum_{\pi \in I_{B}(n,k)}q^{w(\pi )}.
\end{eqnarray*}
Using \eqref{wrib}, we obtain
\begin{eqnarray*}
i_{B_{q}}(n,k) &=&\sum_{j=-n}^{-1}q^{n(n-j-1)}\sum_{\substack{ \tau \in P\left([n]\backslash\{|j|\}\right) 
\\ n-j-1+inv_{B}(\tau )=k}}q^{w(\tau )}+\sum_{j=1}^{n}q^{n
(n-j)}\sum_{\substack{ \tau \in P\left([n]\backslash\{|j|\}\right)  \\ n-j+inv_{B}(\tau )=k}}%
q^{w(\tau )}
\end{eqnarray*}
\begin{eqnarray*}
&=&\sum_{j=n}^{2n-1}q^{nj}\sum_{\substack{ \tau \in P\left([n]\backslash\{|n-j-1|\}\right)  \\ %
j+inv_{B}(\tau )=k}}q^{w(\tau )}+\sum_{j=0}^{n-1}q^{n j}\sum
_{\substack{ \tau \in P\left([n]\backslash\{|n-j|\}\right)  \\ j+inv_{B}(\tau )=k}}%
q^{w(\tau )} \\
&=&\sum_{j=n}^{2n-1}q^{n j}\sum_{\substack{ \tau \in P\left([n]\backslash\{|n-j-1|\}\right)  \\ %
inv_{B}(\tau )=k-j}}q^{w(\tau )}+\sum_{j=0}^{n-1}q^{n j}\sum
_{\substack{ \tau \in P\left([n]\backslash\{|n-j|\}\right)  \\ inv_{B}(\tau )=k-j}}%
q^{w(\tau )} \\
&=&\sum_{j=0}^{2n-1}q^{n j}\sum_{\substack{ \tau \in P\left([n]\backslash\{|n-j|\}\right)  \\ %
inv_{B}(\tau )=k-j}}q^{w(\tau )} \\
&=&\sum_{j=0}^{2n-1}q^{n j}i_{B_{q}}(n-1,k-j)
\end{eqnarray*}
which gives the desired relation.
\end{proof}
We call the table values of $i_{B_q}(n, k)$ as the {\bf $q$-Mahonian triangle of type $B$}. See, Table 2.
\begin{center}
	
	\begin{tabular}{|l|l|l|l|l|l|}
	\hline $n/k$ & $0$ & $1$ & $2$ & $3$ & $4$  \\
	\hline $1$ & $1$ & $q$  & &  &  \\
	\hline $2$ & $1$ & $q^2+q$ & $q^4+q^3$  & $q^6+q^5$ & $q^7$  \\
	\hline $3$ & $1$ & $q^3+q^2+q$ & $q^{6}+q^5+2q^4+q^3$ & $q^{9}+q^8+2q^7+2q^6+q^5$ &$q^{12}+q^{11}+2q^{10}+\cdots$ \\
	
	\hline
	\end{tabular}
	
\end{center}
\begin{center}
	Table 2: The $q$-Mahonian triangle of type $B$.
\end{center}

\bigskip

From Theorem \ref{qrl}, we can deduce the following recurrence relation for the $q$-Mahonian numbers of type $B$.
\begin{corollary}
For $n>1$ and $0\leq k\leq n^2$, we have
\begin{equation}\label{qrecb}
i_{B_q}(n,k)=i_{B_q}(n-1,k)+q^ni_{B_q}(n-1,k-1)-q^{2n^2}i_{B_q}(n-1,k-2n).
\end{equation}
\end{corollary}

The following result is the $q$-analogue of the symmetry relation \eqref{eb1}.
\begin{theorem} For $n\geq 1$ and $0\leq k \leq n^2$, we have
\begin{equation}\label{qrsym}
i_{B_q}(n,k)=q^{\frac{n(n+1)(4n-1)}{6}}i_{B_{\frac{1}{q}}}(n,n^2-k).
\end{equation}
\end{theorem}
\begin{proof}
Let $\pi_0\in B_n$ be the signed permutation which has the maximum number of inversions of type $B$ and then the largest value of the statistic $\omega$. Let $\mathcal{F}:T_1 \rightarrow T_2$  be a mapping where $T_1=\{\pi \in B_n:inv_B(\pi)=k\}$ and $T_2=\{\pi \in B_n:inv_B(\pi)=n^2-k\}$. Thus the mapping $\mathcal{F}:T_1 \rightarrow T_2$, $\mathcal{F}(\pi)=\pi_0\pi$, is bijective since $inv_B(\pi_0\pi)=inv_B(\pi_0)-inv_B(\pi)=n^2-k$ and then $\omega(\pi_0\pi)=\omega(\pi_0)-\omega(\pi)$ by \eqref{statt}.

\medskip
It is easy to see from \eqref{statt} that $\omega(\pi_0)=\sum_{i=1}^n (n+1-i)(2n-2i+1)$. Hence, if $\omega(\pi)=M$, then we have
$$
\omega(\pi_0\pi)=\omega(\pi_0)-M=\frac{n(n+1)(4n-1)}{6}-M.
$$
Therefore, by relation \eqref{eqan} we get the relation \eqref{qrsym}.
\end{proof}
By the recurrence relation \eqref{rqr} of Theorem \ref{qrl}, we  establish that the $q$-Mahonian numbers of type $B$ have the following generating function.
\begin{theorem}\label{gfqm}
The $q$-Mahonian numbers of type $B$ verify, for $n\geq1$, the following product:
\begin{equation}\label{gfq}
\sum_{k=0}^{n^2}i_{B_q}(n,k)z^k=\prod_{j=1}^n\left(1+q^jz+\cdots+(q^j z)^{2j-1}\right).
\end{equation}
\end{theorem}
\begin{proof}
Let $B_n(z;q):=\sum_{k=0}^{n^2}i_{B_q}(n,k)z^k$ be the row generating function of the $q$-Mahonian numbers of type $B$. Then, it follows from  
\eqref{rqr} that  
\begin{align*}
B_n(z;q)&=\sum_{k=0}^{n^2} \sum_{j=0}^{2n-1}q^{nj}i_{B_q}(n-1,k-j)z^k \\
&=\sum_{j=0}^{2n-1}(q^n z)^j\sum_{k=0}^{n^2}i_{B_q}(n-1,k-j)z^{k-j}.
\end{align*}
Since $i_{B_q}(n-1,k-j)=0$ unless $0\leq k-j \leq (n-1)^2$, we have
\begin{align*}
B_n(z;q)&=\sum_{j=0}^{2n-1}(q^n z)^j\sum_{k=j}^{(n-1)^2+j}i_{B_q}(n-1,k-j)z^{k-j}\\
&=\sum_{j=0}^{2n-1}(q^n z)^j\sum_{k=0}^{(n-1)^2}i_{B_q}(n-1,k)z^{k}\\
&=B_{n-1}(z;q)\sum_{j=0}^{2n-1}(q^n z)^j.
\end{align*}
Iterating, we obtain
\begin{align*}
\sum_{k=0}^{n^2}i_{B_q}(n,k)z^k=\prod_{j=1}^n\left(1+q^jz+\cdots+(q^j z)^{2j-1}\right).
\end{align*}
\end{proof}
\subsection{Lattice path and partition/tiling interpretations}
We know from Theorems \ref{thib}, \ref{thib2} and \ref{thib3}, that there is a bijection between the signed permutations of $I_B(n,k)$, the paths of $\mathcal{P}^{B}_{n,k}$, the partitions of $\mathcal{P}r_{n,k}$ and the tilings of $\mathcal{T}^{B}_{n,k}$. This allows us to give in this subsection the path and partition/tiling interpretations of $q$-Mahonian numbers of type $B$.

\medskip

We start by the path interpretation of $i_{B_q}(n,k)$: for each path $P\in \mathcal{P}^{B}_{n,k}$, we denote by $\omega t(P)$ the weight associated to the path $P$ counting the number of boxes above $P$. For example, the weight of the path of Figure \ref{fg1} is $\omega t(P)=1\times 1+2\times 2+3\times 3=14$.

\begin{figure}[h!]

\begin{center}
\begin{tikzpicture}

\draw[step=0.8cm,color=black!30] (-0.81,0) grid (1.6,4.8);
\draw [line width=1pt](-0.8,0) -- (0,0) -- (0,0.8) -- (0.8,0.8) -- (0.8,2.4)--(1.6,2.4)--(1.6,3.2)--(1.6,4)--(1.6,4.8);

\fill[blue] (-0.8,0) circle(1.9pt) ;
\fill[blue] (1.6,4.8) circle(1.9pt) ;
\node[right=0.1pt] at (1.7,4.8){$v$};
\node at (-1.1,0){$u$};
\node at (-1.5,-0.5){$\text{Levels\ }$};
\node at (-0.8,-0.5){$\emptyset$};
\node at (0,-0.5){$1$};
\node at (0.8,-0.5){$2$};
\node at (1.6,-0.5){$3$};

\end{tikzpicture}
\end{center}


\caption{A path $P$ in $\mathcal{P}_{n,k}^I$ of the weight $\omega t(P)=14$.}
  \label{fg1}
\end{figure}
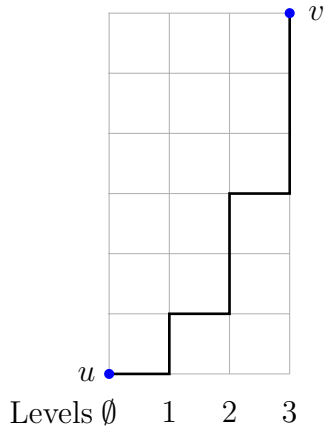
Then, we can interpret $i_{B_q}(n,k)$ by paths as follows.
\begin{theorem}\label{pit} For $n\geq 1$ and $0\leq k\leq n^2$, we have
\begin{equation}\label{qftt}
i_{B_q}(n,k)=\sum_{P\in \mathcal{P}^{B}_{n,k}}q^{\omega t(P)}.
\end{equation}    
\end{theorem}
\begin{proof}
The weight $\omega t(P)$ can be seen as the sum of the product of the number
of vertical steps in each level with the number associated to this level. It is not difficult to see that form each $\pi \in I_B(n,k)$ the weight $\omega(\pi)$ equals to the number of cases above the path $P\in \mathcal{P}^{B}_{n,k}$ associated to $\pi$, i.e., $\omega t(P)=\omega(\pi)$.
\end{proof}

From the previous theorem, we can give the partition interpretation of the $q$-Mahonian
numbers of type $B$ as follows.
\begin{proposition}
For $n\geq 1$ and $0\leq k\leq n^2$, the $q$-Mahonian number of type $B$ is the generating function of the number of partitions into
$k$ parts in which each part $j$ must be used at most $2j-1$ times and the largest part $\leq n$, that is,
\begin{equation}
i_{B_q}(n,k)=\sum_{\lambda\in \mathcal{P}r_{n,k}}q^{\mid \lambda \mid}=\sum_{\lambda \sqsubset n^k}q^{\mid \lambda \mid},
\end{equation} 
where $\mid \lambda \mid =\sum_{j=1}^k \lambda_j$.
\end{proposition}
\begin{proof}
 The exponent of $q$ which is the weight of each path $P\in \mathcal{P}^{B}_{n,k}$ is given by counting
the number of boxes that fit above and to the left of the lattice path, such that
the number of boxes above and to the left of the path in each line represents a
part of the partition as shown in Figure \ref{ff2}. Or, in other words, the number of
vertical steps in level $j$ represents the number of the appearances of the part
$j$.   
\end{proof}

\begin{figure}[h!]
	
	\begin{center}
		\begin{tikzpicture}

		\draw[step=0.8cm,color=black!30] (-0.81,0) grid (1.6,2.4);
		\draw [line width=1pt](-0.8,0) -- (0,0) -- (0,0.8) -- (0.8,0.8) -- (0.8,2.4)--(1.6,2.4);
		
		\fill[blue] (-0.8,0) circle(1.9pt) ;
		\fill[blue] (1.6,2.4) circle(1.9pt) ;
		\node[right=0.1pt] at (1.7,2.4){$v$};
		\node at (-1.1,0){$u$};
		\node at (0.5,-0.5){$\lambda=(1,2,2)$};
	
	\fill[fill=gray!50,draw=black!50] (-0.75,0.05) rectangle (-0.05,0.75);	\fill[fill=gray!50,draw=black!50] (-0.75,0.85) rectangle (-0.05,1.55);
		\fill[fill=gray!50,draw=black!50] (0.05,0.85) rectangle (0.75,1.55);
			\fill[fill=gray!50,draw=black!50] (-0.75,1.65) rectangle (-0.05,2.35);
				\fill[fill=gray!50,draw=black!50] (0.05,1.65) rectangle (0.75,2.35);	
		\end{tikzpicture}
	\end{center}

\caption{The lattice path/partition associated to $q^{|(1,2,2)|}=q^{1+2+2}=q^5$.}
  \label{ff2}
\end{figure}
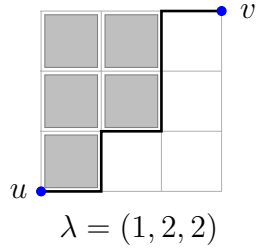
\medskip

For the tiling interpretation of $i_{B_q}(n,k)$, let me give the following notation: for each tiling $T\in \mathcal{T}^{B}_{n,k}$, let $\omega_T$ be the weight of $T$, which is the sum of all the
weights of orange squares in the tiling $T$, where the weight of an orange square is equal
to the number of green squares to the left of that orange square.
\begin{proposition}
For $n\geq 1$ and $0\leq k\leq n^2$, we have
\begin{equation*}
i_{B_q}(n,k)=\sum_{T\in \mathcal{T}^{B}_{n,k}}q^{\omega_T}.
\end{equation*}
\end{proposition}
\begin{proof}
Since the bijection between lattice paths and tiling is weight-preserving. Each green square in the tiling represents an East step of the path, and
each orange square represents a North step, (see Figure \ref{tqq}). It remains to show that
the weight of the tiling and its associated lattice path are the same. In fact,
calculating the weight of the path is the sum of the number of boxes above and
to the left of the path in each row, each North step gives us a row. Calculating
the number of boxes in each row refers to calculate the number of East step
before the North step associated to this row, which is exactly how to calculate
the weight of a tiling.    
\end{proof}
For example, the weight of the tiling 'gogoog' is $q^{1+2+2}=q^5$, as showed in Figure \ref{tqq}.
\begin{figure}[h!]	
	
	\begin{center}
		\begin{tikzpicture}
		
		\draw[step=0.8cm,color=black!30] (0,0) grid(4.8,0.8);
		\fill[fill=green!60,draw=black!50] (0.05,0.05) rectangle (0.75,0.75);
		\fill[fill=orange!90,draw=black!50] (0.85,0.05) rectangle (1.55,0.75);
		\fill[fill=green!60,draw=black!50] (1.65,0.05) rectangle (2.35,0.75);
		\fill[fill=orange!90,draw=black!50] (2.45,0.05) rectangle (3.15,0.75);
		\fill[fill=orange!90,draw=black!50] (3.25,0.05) rectangle (3.95,0.75);
		\fill[fill=green!60,draw=black!50] (4.05,0.05) rectangle (4.75,0.75);
\node[rectangle] at (5.3, 0.4){$q^5$};
		
		\end{tikzpicture}\\
		\vspace{0.1cm}
		\begin{tikzpicture}

		\draw[step=0.8cm,color=black!30] (-0.81,0) grid (1.6,2.4);
		\draw [line width=1pt](-0.8,0) -- (0,0) -- (0,0.8) -- (0.8,0.8) -- (0.8,2.4)--(1.6,2.4);
		
		\fill[blue] (-0.8,0) circle(1.9pt) ;
		\fill[blue] (1.6,2.4) circle(1.9pt) ;
		\node[right=0.1pt] at (1.7,2.4){$v$};
		\node at (-1.1,0){$u$};
		\node at (0.6, -0.5){$q^5$};
		\fill[fill=gray!50,draw=black!50] (-0.75,0.05) rectangle (-0.05,0.75);	\fill[fill=gray!50,draw=black!50] (-0.75,0.85) rectangle (-0.05,1.55);
		\fill[fill=gray!50,draw=black!50] (0.05,0.85) rectangle (0.75,1.55);
		\fill[fill=gray!50,draw=black!50] (-0.75,1.65) rectangle (-0.05,2.35);
		\fill[fill=gray!50,draw=black!50] (0.05,1.65) rectangle (0.75,2.35);

		\end{tikzpicture}
		
	\end{center}
  \caption{Tiling and  lattice path interpretation of $q^5$.}\label{tqq}
\end{figure}

\section{$q$-Log-concavity and log-concavity  properties}\label{qlg}
A sequence of nonnegative numbers $\left(x_{k}\right) _{k}$ is log-concave if $x_{i-1}x_{i+1}\leq x_{i}^{2} \,\text{ for all }i>0$, which is equivalent to $x_{i-1}x_{j+1}\leq x_{i}x_{j}$, for $j\geq i\geq 1$.

\medskip

A finite sequence of real numbers $\left( a_{0},\ldots,a_{m}\right)$ is said to be unimodal if there exists an index $0\leq m^{\ast }\leq m$, such that $a_{0}\leq a_{1}\leq \cdots \leq  a_{m^{\ast }}\geq a_{m^{\ast }+1}\geq \cdots \geq a_{m}$,
the index $m^*$ is called the mode of the sequence.

\medskip

If a sequence of positive real numbers is log-concave then it is unimodal \cite{FB}. Diverse tools and techniques have been developed trying to prove the log-concavity and unimodality properties in different approaches, for more details see \cite{FB1,ST}.

\medskip

Let $ \left( A_{n}(q)\right)_{n}$ be a sequence of polynomials where $q$ is an indeterminate. If, for each $n\geq 1$,  $A_{n}(q)^{2}-A_{n-1}(q)A_{n+1}(q)$ has nonnegative coefficients as polynomial in $q$ (for short, $A_{n}(q)^{2}-A_{n-1}(q)A_{n+1}(q)\geq_q 0$), we say that $\left(A_{n}(q)\right)_{n}$ is $q$-log-concave, and it is strongly $q$-log-concave if $A_{m}(q)A_{n}(q)-A_{m-1}(q)A_{n+1}(q)\geq _{q}0$, for all $n\geq m\geq 1$. For more details, see Stanley \cite{ST}.
	
\bigskip
	
The strong $q$-log-concavity implies the $q$-log-concavity. But the converse is not true in general (see Sagan \cite{Sg}).
       Furthermore, the $q$-log-concavity implies the log concavity and therefore the unimodality for each fixed positive number $q$.
\medskip
	
The first result dealing with the $q$-log-concavity property is due to Butler \cite{BUT} who proved the strong $q$-log-concavity of $q$-binomial coefficients using partitions. Sagan \cite{Sg1} proved the same property for these coefficients using paths. Bazeniar et al. \cite{BZB} studied the same property for the $q$-bi$^s$nomial coefficients by a similar approach to that of Sagan. Dousse and Kim \cite{DK2} generalized this property for the over analogues of $q$-binomial coefficients by a similar approach to that of Butler. Recently, Ghemit and Ahmia in \cite{GA1,GA2} established this property for the over analogues of $q$-bi$^s$nomial coefficients and the $q$-analogue of Mahonian numbers, using respectively paths/partitions and the injection property over the set of permutations.

\medskip
Motivated by these works, we prove in this section that the $q$-Mahonian numbers of type $B$ form a strongly $q$-log-concave sequence of polynomials in $k$, by using the path approach due to Sagan \cite{Sg1}, which implies that the Mahonian numbers of type $B$ form a log-concave sequence in $k$, and therefore unimodal.
\bigskip

\medskip

According to Theorem \ref{pit}, we have 

\begin{equation}
i_{B_q}(n,k)=\sum_{P\in \mathcal{P}^{B}_{n,k}}q^{\omega t(P)}
\end{equation} 
and the weight $wt(P)$ equals to the number of cases above the path $P$ $\in \mathcal{P}^{B}_{n,k}$ with $P$ starts from $u_{1}=(0,0)$ to $v_{1}=(n,k)$, then
\begin{equation}\label{RRGH}
(i_{B_q}(n,k))^2=\sum_{P_{1}\in \mathcal{P}^{B}_{n,k}}q^{\omega t(P_{1})}\sum_{P_{2}\in \mathcal{P}^{B}_{n,k}}q^{\omega t(P_{2})},
\end{equation} 
where $P_{1}$ starts from $u_{1}=(0,0)$ to $v_{1}=(n,k)$ and $P_{2}$ starts from $u_{2}=(0,1)$ to $v_{1}=(n,k+1)$.
The previous equation is equivalent to 

\begin{equation}
(i_{B_q}(n,k))^2=\sum_{P_{1},P_{2}\in \mathcal{P}^{B}_{n,k}}q^{\omega t(P_{1}P_{2})},
\end{equation} 
where $\omega t(P_{1}P_{2})=\omega t(P_{1})+\omega t(P_{2})$.

\medskip
Let $P_1,P_2\in P_{n,k}^B$. And let $u \overset{{P}}{\longrightarrow} v$ denote  $P$ has initial vertex $u$ and final vertex $v$.

\medskip

Using similar technic of Sagan in \cite{Sg1} to prove the strong $q$-log-concavity of the $q$-Mahonian numbers of type $B$, we recall then the following definition used by him.
\begin{definition}\label{dfinv}
	Given $u_1 \overset{{P_1}}{\longrightarrow} v_1$ and $u_2 \overset{{P_2}}{\longrightarrow} v_2$. Then define the involution $\varphi_I(P_1,P_2)=(P^{'}_{1},P^{'}_{2})$
	where
	\begin{enumerate}
\item[] $$P^{'}_{1}=u_1 \overset{{P_1}}{\longrightarrow} v_0 \overset{{P_2}}{\longrightarrow} v_2 \text{\ and\ }P^{'}_{2}= u_2 \overset{{P_2}}{\longrightarrow} v_0 \overset{{P_1}}{\longrightarrow} v_1,$$
i.e., switches the portions of
		$P_1$ and $P_2$ after $v_0$ (see, Figure \ref{Fig2}), where $v_0$ is the last vertex of $P_1 \cap P_2$ such that the number of vertical steps,
in the vertical level of $v_0$, of $P_1$ and $P_2$ all together does not exceed $2m-1$ if the vertical level of $v_0$ is $m$.	
	\end{enumerate}
	
\end{definition}
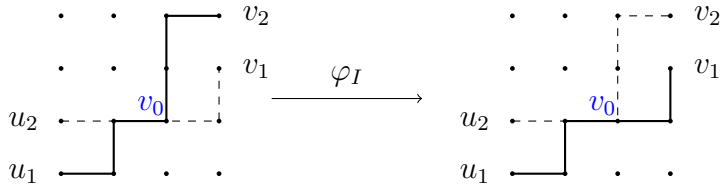
\begin{figure}[h!]
		\begin{center}
			
			\begin{tikzpicture}
			
			\draw[dashed]  (-0.7,0.7)--(0,0.7)--(0.7,0.7)--(1.4,0.7)--(1.4,1.4);
			\draw [line width=0.8pt] (-0.7,0)--(0,0)--(0,0.7)--(0.7,0.7)--(0.7,1.4)--(0.7,2.1)--(1.4,2.1);
			
			\fill[] (-0.7,0) circle (0.9pt);
			\fill[] (-0.7,0.7) circle (0.9pt);
			\fill[] (-0.7,1.4) circle (0.9pt);
			\fill[] (-0.7,2.1) circle (0.9pt);

			\fill[] (0,0) circle (0.9pt);
			\fill[] (0,0.7) circle (0.9pt);
			\fill[] (0,1.4) circle (0.9pt);
			\fill[] (0,2.1) circle (0.9pt);

			\fill[] (0.7,0) circle (0.9pt);
			\fill[] (0.7,0.7) circle (0.9pt);
			\fill[] (0.7,1.4) circle (0.9pt);
			\fill[] (0.7,2.1) circle (0.9pt);

			\fill[] (1.4,0) circle (0.9pt);
			\fill[] (1.4,0.7) circle (0.9pt);
			\fill[] (1.4,1.4) circle (0.9pt);
			\fill[] (1.4,2.1) circle (0.9pt);


			\node at (-1.2,0){$u_1$};
			\node at (-1.2,0.7){$u_2$};
			\node at (1.9,1.4){$v_1$};
			\node at (1.9,2.1){$v_2$};
			
			\node at (0.5,0.9){$\textcolor[rgb]{0.00,0.00,1.00}{v_0}$};

			\begin{scope}[xshift=1.9cm]
			\draw [->] (0.2,1)--(2.2,1);
			\node[rectangle] at (1.2,1.3) {$\varphi_I$};
			
			\begin{scope}[xshift=4.1cm]
			\draw[dashed]    (-0.7,0.7)--(0,0.7)--(0.7,0.7)--(0.7,1.4)--(0.7,1.4)--(0.7,2.1)--(1.4,2.1);
			\draw [line width=0.8pt] (-0.7,0)--(0,0)--(0,0.7)--(0.7,0.7)--(1.4,0.7)--(1.4,1.4);
			
			\fill[] (-0.7,0) circle (0.9pt);
			\fill[] (-0.7,0.7) circle (0.9pt);
			\fill[] (-0.7,1.4) circle (0.9pt);
			\fill[] (-0.7,2.1) circle (0.9pt);

			\fill[] (0,0) circle (0.9pt);
			\fill[] (0,0.7) circle (0.9pt);
			\fill[] (0,1.4) circle (0.9pt);
			\fill[] (0,2.1) circle (0.9pt);

			\fill[] (0.7,0) circle (0.9pt);
			\fill[] (0.7,0.7) circle (0.9pt);
			\fill[] (0.7,1.4) circle (0.9pt);
			\fill[] (0.7,2.1) circle (0.9pt);

			\fill[] (1.4,0) circle (0.9pt);
			\fill[] (1.4,0.7) circle (0.9pt);
			\fill[] (1.4,1.4) circle (0.9pt);
			\fill[] (1.4,2.1) circle (0.9pt);


			\node at (-1.2,-0){$u_1$};
			\node at (-1.2,0.7){$u_2$};

		    \node at (1.9,1.4){$v_1$};
			\node at (1.9,2.1){$v_2$};

			\node at (0.5,0.9){$\textcolor[rgb]{0.00,0.00,1.00}{v_0}$};

			\end{scope}
			\end{scope}
			\end{tikzpicture}	
		\end{center}
		
		\caption{The involution $\varphi_{I}$.}\label{Fig2}
	\end{figure}

Now, we give the following theorem.
\begin{theorem}\label{tmr}
For $n$ fixed, the sequence of polynomials $\left( i_{B_q}(n,k)\right)_{0\leq k \leq n^2}$ is strongly q-log concave in $k$, that is,
\[i_{B_q}(n,l)\times i_{B_q}(n,k) -i_{B_q}(n,l-1)\times i_{B_q}(n,k+1)\geq_{q}0,\]
where $0<l\leq k$.
\end{theorem}
\begin{proof}
First, we will show that
\begin{equation}\label{inequ1}
\left( i_{B_q}(n,k)\right)^2-i_{B_q}(n,k-1)\times i_{B_q}(n,k+1)\geq_{q}0.
\end{equation}
By relation \eqref{RRGH}, the left-hand side of the inequality \eqref{inequ1} can be written as follows
\begin{equation}\label{equpath}
\left( i_{B_q}(n,k)\right)^2- i_{B_q}(n,k-1)\times i_{B_q}(n,k+1)=\sum_{P_1 , P_2}(-1)^{P_1 P_2}q^{\omega t(P_1P_2)},
\end{equation}
where the sum is over all pairs $P_1, P_2$ such that $P_1$ starts at $u_1 =(0,0)$, $P_2$ starts at $u_2 =(0,1)$, and

$$(-1)^{P_1 P_2}=\left\{\begin{array}{cccc}
+1 & \text{\ if\ } u_1 \overset{{P_1}}{\longrightarrow} v_1 \hspace{0.3cm} \text{and} \hspace{0.3cm}u_2 \overset{{P_2}}{\longrightarrow} v_2,\\
-1 & \text{\ if\ }u_1 \overset{{P_1}}{\longrightarrow} v_2 \hspace{0.3cm} \text{and} \hspace{0.3cm}u_2 \overset{{P_2}}{\longrightarrow} v_1,\\\end{array}\right.$$
with $v_1 =(n,k)$ and $v_2=(n,k+1)$.

\medskip

On each path pair $(P_1,P_2)$  with sign $-1$ in equation (\ref{equpath}), we apply the involution $\varphi_I$ given in Definition \ref{dfinv}, we obtain a path pair $(P'_1, P'_2)$ with sign $+1$, and since $P_1$ and $P_2$ start on the same vertical line, the sum of the number of boxes above $P_1^{'}$ and the
number of boxes above $P'_2$ after switching remains the same i.e., $\omega t(P_1 P_2)= \omega t(P'_1P'_2)$ (see,  Figure \ref{Fig2}), because the boxes lost by the first are gained by the second. Thus the inequality (\ref{inequ1}) is verified.

\medskip

To prove the general case 
\begin{equation}
 i_{B_q}(n,l)\times i_{B_q}(n,k)-i_{B_q}(n,l-1)\times i_{B_q}(n,k+1)\geq_{q}0,
\end{equation}
where $ 0< l\leq k$, it suffices to use the same approach with $u_2=(0,k-l+1)$ as new condition.

\end{proof}
As an illustration, we give the following example.
\begin{example}
For $n=2$ and $k=2$, we have
\medskip
\begin{align*}
&(i_{B_q}(2,2))^2=(q^4+q^3)^2=q^8+2q^7+q^6,\\
&i_{B_q}(2,1).i_{B_q}(2,3)=(q^2+q).(q^6+q^5)=q^8+2q^7+q^6.
\end{align*}
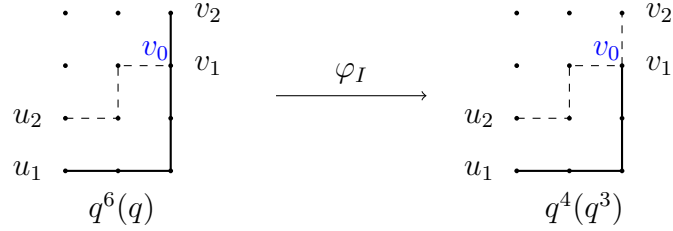
\begin{figure}[h!]
		\begin{center}
			
			\begin{tikzpicture}
			
			\draw[dashed]  (-0.7,0.7)--(0,0.7)--(0,1.4)--(0.7,1.4);
			\draw [line width=0.8pt] (-0.7,0)--(0,0)--(0.7,0)--(0.7,0.7)--(0.7,1.4)--(0.7,2.1);
			
			\fill[] (-0.7,0) circle (0.9pt);
			\fill[] (-0.7,0.7) circle (0.9pt);
			\fill[] (-0.7,1.4) circle (0.9pt);
			\fill[] (-0.7,2.1) circle (0.9pt);

			\fill[] (0,0) circle (0.9pt);
			\fill[] (0,0.7) circle (0.9pt);
			\fill[] (0,1.4) circle (0.9pt);
			\fill[] (0,2.1) circle (0.9pt);

			\fill[] (0.7,0) circle (0.9pt);
			\fill[] (0.7,0.7) circle (0.9pt);
			\fill[] (0.7,1.4) circle (0.9pt);
			\fill[] (0.7,2.1) circle (0.9pt);
			\node at (-1.2,0){$u_1$};
   	\node at (0.05,-0.5){$q^6(q)$};
			\node at (-1.2,0.7){$u_2$};
			\node at (1.2,1.4){$v_1$};
			\node at (1.2,2.1){$v_2$};
			
			\node at (0.5,1.6){$\textcolor[rgb]{0.00,0.00,1.00}{v_0}$};

			\begin{scope}[xshift=1.9cm]
			\draw [->] (0.2,1)--(2.2,1);
			\node[rectangle] at (1.2,1.3) {$\varphi_I$};
			
			\begin{scope}[xshift=4.1cm]
			\draw[dashed]  (-0.7,0.7)--(0,0.7)--(0,1.4)--(0.7,1.4)--(0.7,2.1);
			\draw [line width=0.8pt] (-0.7,0)--(0,0)--(0.7,0)--(0.7,0.7)--(0.7,1.4);
			
			\fill[] (-0.7,0) circle (0.9pt);
			\fill[] (-0.7,0.7) circle (0.9pt);
			\fill[] (-0.7,1.4) circle (0.9pt);
			\fill[] (-0.7,2.1) circle (0.9pt);

			\fill[] (0,0) circle (0.9pt);
			\fill[] (0,0.7) circle (0.9pt);
			\fill[] (0,1.4) circle (0.9pt);
			\fill[] (0,2.1) circle (0.9pt);

			\fill[] (0.7,0) circle (0.9pt);
			\fill[] (0.7,0.7) circle (0.9pt);
			\fill[] (0.7,1.4) circle (0.9pt);
			\fill[] (0.7,2.1) circle (0.9pt);
			
			\node at (-1.2,0){$u_1$};
   \node at (0.2,-0.5){$q^4(q^3)$};
			\node at (-1.2,0.7){$u_2$};
			\node at (1.2,1.4){$v_1$};
			\node at (1.2,2.1){$v_2$};
			
			\node at (0.5,1.6){$\textcolor[rgb]{0.00,0.00,1.00}{v_0}$};

			\end{scope}
			\end{scope}
			\end{tikzpicture}	
		\end{center}
		
		\caption{The application of the involution $\varphi_{I}$.}\label{Fig3}
	\end{figure}


\medskip	

Applying $\varphi_I$ on the paths $P_1$ ($u_1 \overset{{P_1}}{\longrightarrow} v_2$) and $P_2$ ($u_2 \overset{{P_2}}{\longrightarrow} v_1$), as shown in Figure \ref{Fig3}, associated to the term  of  $i_{B_q}(2,3)\times i_{B_q}(2,1)$ : $q^6$ and $q$ respectively, we obtain the paths  $P'_1$ ($u_1 \overset{{P'_1}}{\longrightarrow} v_1$) and $P'_2$ ($u_2 \overset{{P'_2}}{\longrightarrow} v_2$) associated to the term  of $(i_{B_q}(2,2))^2$: $q^4$ and $q^3$ respectively.
\end{example}

By setting $q=1$ in Theorem \ref{tmr}, we can establish that the Mahonian numbers of type $B$ satisfy the following results about the log-concavity and unimodality.
\begin{corollary}
The Mahonian numbers $(i_B(n,k))_{0\leq k \leq n^2}$ of type $B$ form a log-concave sequence in $k$, and therefore unimodal.   
\end{corollary}

\medskip

\section{Concluding remark and questions} \label{QR}
In the previous section, we proved that the $q$-analogue of Mahonian numbers of type $B$ form a strongly $q$-log-concave sequence of polynomials in $k$. But in $n$, we have not been able to prove it by the same approach, we found that we can not apply the involution of Definition \ref{dfinv} in the proof, because the number of the vertical steps after the switching in the level $m$ of $v_0$ exceed $2m-1$,  which contradicts the conditions of Definition \ref{dfinv}. Numerically, the sequence (resp. the sequence of polynomials) $\left( i_{B}(n,k)\right)_{n}$ (resp.  $\left( i_{B_q}(n,k)\right)_{n}$) is log-concave (resp. strongly $q$-log concave) in $n$, so we propose the following  question. 

\medskip
 
\noindent{\bf Question 1.}
For $k$ fixed, is the sequence (resp. the sequence of polynomials) $\left( i_{B}(n,k)\right)_{n}$ (resp.  $\left( i_{B_q}(n,k)\right)_{n}$) log-concave (resp. strongly $q$-log concave) in $n$? 

\medskip
Moreover, the sequence $(i_B(n,k))_{k}$ is unimodal. However the number and location of the modes
of this sequence remains a question to be answered. But the answer to this question is not easy to find, so we propose it as follows.

\medskip

\noindent{\bf Question 2.} Find the number and location of modes of the unimodal sequence $(i_B(n,k))_{k}$.

\section*{Acknowledgement}
The authors would like to thank the referees for many valuable remarks and suggestions to
improve the original manuscript. This work was supported by DG-RSDT (Algeria), PRFU Project C00L03UN180120220002 and PRFU Project C00L03UN190120230012..

\end{document}